\documentclass[preprint,11pt]{elsarticle}

\usepackage{amsmath,amsthm,amssymb,booktabs,url}

\usepackage{enumitem}
\setlist[enumerate]{label=\roman*),itemsep=0pt}

\newtheorem{thm}{Theorem}
\newtheorem{lem}[thm]{Lemma}

\newtheorem{obs}[thm]{Observation}
\newtheorem{coro}[thm]{Corollary}
\newtheorem{prop}[thm]{Proposition}

\theoremstyle{definition}

\newtheorem{df}[thm]{Definition}

\theoremstyle{remark}

\newtheorem{ex}[thm]{Example}

\newcommand{\PL}[1]{\ensuremath{\mathrm{PL}_{\boldsymbol{#1}}}}
\newcommand{\bu}{\ensuremath{\boldsymbol{u}}}
\newcommand{\pall}[1]{\ensuremath{|#1|_{\mathrm{pal}}}}
\newcommand{\pallc}[1]{\ensuremath{|#1|_{\mathrm{cpal}}}}

\newcommand{\N}{\mathbb{N}}

\newcommand{\Lc}{\mathcal{L}}
\newcommand{\Pc}{\mathcal{P}}

\begin{document}

\begin{frontmatter}

\title{Palindromic length of words and morphisms in class $\mathcal{P}$}

\author[fjfi]{Petr Ambro\v{z}\corref{cor1}}
\ead{petr.ambroz@fjfi.cvut.cz}

\author[fjfi]{Ond\v{r}ej Kadlec}
\ead{kadleond@fjfi.cvut.cz}

\author[fjfi]{Zuzana Mas\'akov\'a}
\ead{zuzana.masakova@fjfi.cvut.cz}

\author[fjfi]{Edita Pelantov\'a}
\ead{edita.pelantova@fjfi.cvut.cz}

\cortext[cor1]{Corresponding author}

\address[fjfi]{FNSPE, Czech Technical University in Prague, Trojanova 13, 120 00 Praha 2, Czech~Republic}

\begin{abstract}
We study the palindromic length of factors of infinite words fixed by morphisms of the so-called class $\Pc$
introduced by Hof, Knill and Simon. We show that it grows at most logarithmically with the length of the factor.
For the Fibonacci word and the Thue-Morse word we provide estimates on the constants of the growth.
We also construct an infinite word rich in palindromes for which the palindromic length grows as $\sqrt{n}$.
\end{abstract}


\end{frontmatter}

%%%%%%%%%%%%%%%%%%%%%%%%%%%%%%%%%%%%%%%%%%%%%%%%%%%%%%%%%%%%%%%%%%%%%%%%%%%%%%%%%%%%%%%%%%%%%%%%%%%%%%

\section{Introduction}

In this paper we show a connection between two conjectures concerning palindromes in languages of infinite
words. A palindrome is a word which reads the same forward as backward, such as eye or kayak. An impulse
to formulate the first conjecture comes from paper by Hof, Knill and Simon~\cite{hof-knill-simon-cmp-174},
where they studied infinite words generated by primitive morphisms.
They defined a certain class of
morphisms, called class $\Pc$, by requiring that the image of any letter $a$ in the alphabet is in the form
$a\mapsto p q_a$, where $p$ and $q_a$ are palindromes, $p$ being common to every $a$.
They showed that fixed points of morphisms in this class contain an infinite
number of palindromes. They asked whether any palindromic fixed point of a primitive
substitution arises using such a morphism.
This question was eventually turned into a conjecture (called later the
HKS conjecture), however,
due to a certain vagueness of the original question several versions of this conjecture have been
considered (for details see Introduction in~\cite{labbe-pelantova-ejc-51}). Validity of HKS conjecture was
proved for binary words~\cite{tan-tcs-389} and for words fixed by a morphism of certain
types~\cite{labbe-pelantova-ejc-51,masakova-pelantova-starosta-ejc-62}.

In 2013, Frid, Puzynina and Zamboni~\cite{frid-puzynina-zamboni-aam-50} introduced the palindromic length
of a finite word $w$, denoted by $\pall{w}$, as the minimal number of palindromes whose concatenation
is equal to $w$. They conjectured that if there is a constant $K$ such that the palindromic length
of every factor in an infinite word $\bu$ is bounded by $K$, then $\bu$ is eventually periodic.
Formally, defining for a given infinite word $\bu$ the function $\PL{u}:\N\to\N$ by
\[
\PL{u}(n) := \max \{ \pall{w} : \text{$w$ is a factor of length $n$ in \bu} \},
\]
the conjecture is that if $\bu$ is not eventually periodic, then $\limsup_{n\to\infty} \PL{u}(n) =+\infty$.
The authors of~\cite{frid-puzynina-zamboni-aam-50} proved the conjecture for infinite words which
do not contain an $r$-power for any positive integer $r$. In particular, the conjecture is true for any
aperiodic fixed point of a primitive morphism, as such fixed points have bounded powers~\cite{mosse-tcs-99}.
Later, Frid~\cite{frid-ejc-71} showed that
Sturmian words have unbounded palindromic length even if they contain unbounded powers.
Palindromic length of Sturmian words is studied also in~\cite{ambroz-pelantova-preprint}.
It is shown that $\PL{u}$ can grow arbitrarily slowly.
For other infinite words besides Sturmian words and bounded-repetition words the conjecture of
$\limsup_{n\to\infty}\PL{u}(n)$ $= +\infty$ remains open.

\medskip

We study palindromic length of factors of fixed points of primitive morphisms. Here, as we have stated above,
the palindromic length in unbounded, whenever the fixed point is not eventually periodic.
The main results of this contribution are formulated as Proposition~\ref{prop:upper_est_on_pall} and
Theorem~\ref{thm:hlavni}.
We prove that if the HKS conjecture is valid then for any primitive morphism $\varphi$
there is a constant $K>0$ such that the palindromic length of every factor $w$ in the language of $\varphi$
is less than or equal to $K\ln|w|$.
We also provide a method of estimating the constant $K$.

For the case of the Fibonacci word $\boldsymbol{f}$ our computations suggest that
$\limsup_{n\to\infty}\frac{\PL{f}(n)}{\ln n}\leq \frac{2}{{3}\ln \tau}$,
where $\tau=\frac12(1+\sqrt{5})$ is the golden ratio. We give an upper bound on $\limsup_{n\to\infty}\frac{\PL{u}(n)}{\ln n}$
also for the Thue-Morse word. The estimates are given in Section~\ref{sec:FibThu}.
Let us mention that a lower bound on $\limsup\PL{u}$ is not known
even for the Fibonacci word. Frid~\cite{frid-numeration-2018} conjectures that
$\limsup_{n\to\infty}\frac{\PL{f}(n)}{\ln n} \geq \frac{1}{3\ln \tau}$.

The Fibonacci word $\boldsymbol{f}$ belongs among the so-called rich words (or full words)
introduced in~\cite{droubay-justin-pirillo-tcs-255,brlek-hamel-nivat-reutenauer-ijfcs-15}. An infinite word is rich
if each of its finite factors is rich, i.e., contains as many palindromes as possible. Intuitively, the palindromic length of
of such words should grow slowly. We demonstrate that this is not necessarily the case.
In Section~\ref{sec:rich}, we use finite rich words introduced by Guo, Shallit and Shur~\cite{guo-shallit-shur-ipl-116}
to construct an infinite word $\bu$ for which $\PL{u}(n)\geq c\sqrt{n}$ for some positive constant $c$.

Section~\ref{sec:antipal} is devoted to further computer experiments on the Thue-Morse word $\boldsymbol{t}$ whose
language contains besides palindromes also infinitely many antipalindromes. We introduce the combined pal-antipal length
and compare its growth to the growth of $\PL{t}$.

In the final section we formulate several open questions concerning factorization to palindromes and/or antipalindromes.

%%%%%%%%%%%%%%%%%%%%%%%%%%%%%%%%%%%%%%%%%%%%%%%%%%%%%%%%%%%%%%%%%%%%%%%%%%%%%%%%%%%%%%%%%%%%%%%%%%%%%%

\section{Preliminaries}

Let $A$ be a finite set called \emph{alphabet}, its elements are called \emph{letters}.
A \emph{word} $w=w_1\cdots w_n$ (over $A$) is a finite sequence of elements in $A$, its \emph{length}
(the number of its elements) is denoted by $|w|=n$. The notation $|w|_a$ is used for the number of
occurrences of the letter $a$ in $w$. The \emph{empty word} -- unique word of length zero --
is denoted by $\varepsilon$. The \emph{concatenation} of words $v=v_1\cdots v_k$ and $w=w_1\cdots w_l$ is
$vw=v\cdot w=v_1\cdots v_kw_1\cdots w_l$. The set of all finite words over $A$ equipped with the operation
concatenation of words is a free monoid, denoted by $A^*$.

For a word $w=w_1\cdots w_n$ we define its mirror image as $\overleftarrow{w}=w_n\cdots w_1$.
A word $w$ is called \emph{palindrome} if $w=\overleftarrow{w}$. The \emph{palindromic length} of
a word $w$, denoted by $\pall{w}$, is the smallest number $K$ of palindromes $p_1,\ldots,p_K$ such that
$w=p_1\cdots p_K$, i.e., the minimal number of palindromes whose concatenation is equal to $w$. For convenience,
we define $\pall{\varepsilon}=0$.

An infinite sequence of letters $\bu=(u_i)_{i\geq 1}$ in $A$ is called \emph{infinite word}. The set of all
infinite words over $A$ is denoted $A^{\N}$. The word $\bu\in A^{\N}$ is said to be \emph{eventually periodic}
if it is of the form $\bu=vz^{\omega}$, where $v,z\in A^*$, $z\neq\varepsilon$ and $z^{\omega}=zzz\cdots$.

A \emph{factor} of a (finite of infinite) word $w$ is a finite word $v$ such that $w=w_1vw_2$ for some
words $w_1,w_2$. If $w_1=\varepsilon$ then $v$ is called a \emph{prefix} of $w$, if $w_2=\varepsilon$ then
$v$ is called a \emph{suffix} of $w$. The set of all factors of an infinite word $\bu$, called the
\emph{language} of $\bu$, is denoted by $\Lc(\bu)$. Let $v_p$ be a prefix of a word $v$, that is,
there is a word $x$ such that $v=v_px$. Then we define $v_p^{-1}v = v_p^{-1}\cdot v = x$.
Similarly, let $v_s$ be a suffix of a word $v$, that is,
there is a word $y$ such that $v=yv_s$. Then we define $vv_s^{-1} = v\cdot v_s^{-1} = y$.

Let $w$ be a finite word over the alphabet $A=\{a_1,\ldots,a_r\}$. Then the \emph{Parikh vector} of $w$
is the vector, denoted $\vec{V}(w)$, whose $i$-th element is the number of occurrences  of
$a_i$ in $w$, i.e.,
\[
\vec{V}(w)=\begin{pmatrix}|w|_{a_1} \\ \vdots \\ |w|_{a_r}\end{pmatrix}.
\]
Obviously, $|w|=\boldsymbol{1}\cdot\vec{V}(w)$, where $\boldsymbol{1}=(1\ 1 \cdots 1)$.

A \emph{morphism} of the free monoid $A^*$ is a map $\varphi:A^*\to A^*$ such that
$\varphi(vw)=\varphi(v)\varphi(w)$ for all $v,w\in A^*$. A morphism of $A^*$, where $A=\{a_1,\ldots,a_r\}$,
is called \emph{primitive} if there is a constant $K$ such that $\varphi^K(a_i)$ contains $a_j$ for
every $i,j\in\{1,\ldots,r\}$. The action of the morphism $\varphi$ is naturally extended to infinite words by concatenation,
in particular, we have
\[
\varphi(u_0u_1u_2\cdots) = \varphi(u_0)\varphi(u_1)\varphi(u_2)\cdots
\]
An infinite word $\bu$ is called a \emph{fixed point} of the morphism $\varphi$ if $\bu=\varphi(\bu)$.
Clearly, a morphism $\psi$ can have several different fixed points, however, if $\psi$ is primitive then
all its fixed points have the same language, denoted $\Lc(\psi)$.

Let $A=\{a_1,\ldots,a_r\}$ and let $\psi$ be a morphism of $A^*$. The \emph{incidence matrix} of $\psi$ is the
$r\times r$ matrix $\boldsymbol{M}_{\!\psi}$ given by $[\boldsymbol{M}_{\!\psi}]_{ij}=|\psi(a_j)|_{a_i}$. The
incidence matrix of $\psi$ can be used to compute the Parikh vector of the image of a word $w$ under $\psi$ by
\begin{equation}\label{eq:parikh_vector_of_image_via_incid_matrix}
\vec{V}(\psi(w)) = \boldsymbol{M}_{\!\psi}\cdot\vec{V}(w).
\end{equation}

%XXXXX Pridat, jak roste delka slova pomoci vlastniho cisla.

%%%%%%%%%%%%%%%%%%%%%%%%%%%%%%%%%%%%%%%%%%%%%%%%%%%%%%%%%%%%%%%%%%%%%%%%%%%%%%%%%%%%%%%%%%%%%%%%%%%%%%
\section{Morphisms in class $\Pc$}

\begin{df}
A primitive morphism $\psi:A^*\to A^*$ belongs to class $\Pc$ if there is a palindrome $p\in A^*$
such that for each $a\in A$
\begin{equation}\label{eq:def_class_P}
\psi(a) = pq_a, \qquad \text{where $q_a\in A^*$ is a palindrome.}
\end{equation}
\end{df}

\begin{ex}\label{ex:fibmorphism}
The Fibonacci morphism $\varphi_F: a\mapsto ab, b\mapsto a$ belongs to  class $\Pc$;
Equation~(\ref{eq:def_class_P}) is fulfilled for $p=a, q_a=b, q_b=\varepsilon$.
\end{ex}

\begin{ex}\label{ex:tmmorphisms}
The Thue-Morse morphism $\varphi_{\text{TM}}: a\mapsto ab, b\mapsto ba$ does not belong to class $\Pc$,
however, its square $\varphi_{\text{TM}}^2: a\mapsto abba, b\mapsto baab$ does ($p=\varepsilon, q_a=abba, q_b=baab$).
\end{ex}

The following simple observation is due to Hof, Knill and Simon~\cite{hof-knill-simon-cmp-174}.

\begin{obs}\label{obs:image_under_classP_morphism}
Let $\psi$ be a primitive morphism in the form~(\ref{eq:def_class_P}) and let $\bu$ be a fixed point of $\psi$.
Then
\begin{enumerate}
\item
  if $w\in\Lc(\psi)$ then $\psi(w)p\in\Lc(\psi)$,
\item
  if $w$ is a palindrome then $\psi(w)p$ is a palindrome.
\end{enumerate}
\end{obs}

By repeated application of Observation~\ref{obs:image_under_classP_morphism}, one obtains the following corollary,
which was first shown in~\cite{hof-knill-simon-cmp-174}.

\begin{coro}\label{c:Ppalindromic}
The language of a fixed point of a morphism in class $\Pc$ contains infinitely many palindromes.
\end{coro}

In fact, as was noticed in~\cite{allouche-baake-cassaigne-damanik-tcs-292}, the same statement as Corollary~\ref{c:Ppalindromic}
is valid for fixed points of morphisms that are not in class $\Pc$ by themselves, but some of their conjugates is, see
Definition~\ref{d:conjugace} below. The reason for that is that languages of infinite words fixed by conjugated primitive
morphisms coincide.

\begin{df}\label{d:conjugace}
Morphisms $\psi_1,\psi_2:A^*\to A^*$ are said to be \emph{conjugated}, denoted by $\psi_1\sim\psi_2$,
if there is a word $w\in A^*$ such that either $\psi_1(a)w=w\psi_2(a)$ for every $a\in A$ or
$w\psi_1(a)=\psi_2(a)w$ for every $a\in A$.
\end{df}

The proof of the following fact can be found for example in~\cite{allouche-baake-cassaigne-damanik-tcs-292}.

\begin{prop}\label{p:conjuglang}
Let $\psi_1$ be a primitive morphism and let $\psi_2$ be conjugated with $\psi_1$. Then
$\Lc(\psi_1)=\Lc(\psi_2)$.
\end{prop}

In~\cite{allouche-baake-cassaigne-damanik-tcs-292} the authors also make the observation that any morphism of
class $\Pc$ is conjugated to a morphism of the form~\eqref{eq:def_class_P} where the palindrome $p$ is either the empty
word or a single letter. From now on, in view of Proposition~\ref{p:conjuglang}, we will only consider morphisms in class
$\Pc$ of this form. The following lemma shows how palindromic length of a finite word changes under application of such
a morphism.

\begin{lem}\label{lem:pall_of_image_for_p_shorter_than_2}
Let $\psi: A^*\to A^*$ be a morphism in class $\Pc$ in the form~\eqref{eq:def_class_P}.
\begin{enumerate}
\item
  If $p=\varepsilon$, then $\pall{\psi(w)}\leq\pall{w}$ for every $w\in A^*$.
\item
  Suppose $|p|=1$. If $\pall{w}$ is even then $\pall{\psi(w)}\leq\pall{w}$, otherwise $\pall{\psi(w)}\leq\pall{w}+1$.
\item
  If $\pall{\psi(w)}=\pall{w}+1$ then $\pall{\psi^2(w)}\leq\pall{\psi(w)}$.
\end{enumerate}
\end{lem}

\begin{proof}
i) Let $w=p_1\cdots p_k$, where $p_1,\ldots,p_k$ are palindromes. Then by
Observation~\ref{obs:image_under_classP_morphism},
$\psi(w)=\psi(p_1)\cdot\psi(p_2)\cdots\psi(p_k)$ is a concatenation of $\psi(w)$ into $k$ palindromes.

ii) Let $w=p_1\cdots p_{2k}$, where  $p_1,\ldots,p_{2k}$ are palindromes. Then by
Observation~\ref{obs:image_under_classP_morphism},
\begin{equation}\label{eq:psi_w2k_into_pal}
\psi(w) = \underbrace{\psi(p_1)p}_{q_1}\cdot\underbrace{p^{-1}\psi(p_2)}_{q_2}\cdot
\underbrace{\psi(p_3)p}_{q_3}\cdot\underbrace{p^{-1}\psi(p_4)}_{q_4}\cdots
\underbrace{\psi(p_{2k-1})p}_{q_{2k-1}}\cdot\underbrace{p^{-1}\psi(p_{2k})}_{q_{2k}}
\end{equation}
is a concatenation of $\psi(w)$ into $2k$ palindromes $q_1,\ldots,q_{2k}$.
If $w=p_1\cdots p_{2k+1}$, where  $p_1,\ldots,p_{2k+1}$ are palindromes then
$\psi(w)$ can be factorized into $2k+2$ palindromes similarly:
\begin{equation}\label{eq:psi_w2k+1_into_pal}
\psi(w) = \underbrace{\psi(p_1)p}_{q_1}\cdot\underbrace{p^{-1}\psi(p_2)}_{q_2}\cdots
\underbrace{\psi(p_{2k-1})p}_{q_{2k-1}}\cdot\underbrace{p^{-1}\psi(p_{2k})}_{q_{2k}}\underbrace{\cdot\ p\ \cdot}_{q_{2k+1}}
\underbrace{p^{-1}\psi(p_{2k+1})}_{q_{2k+2}}.
\end{equation}

iii) By i) and ii), $\pall{\psi(w)}=\pall{w}+1$ may happen only if $\pall{w}$ is odd. Then $\pall{\psi(w)}$ is even and thus
$\pall{\psi^2(w)}\leq \pall{\psi(w)}$.
\end{proof}

The above lemma states that by applying a morphism $\psi$ of class $\Pc$ to a word, palindromic length can increase
by at most one, and this happens only at alternating iterations of the morphism $\psi$. With this knowledge, we can
find an estimate on the growth of the palindromic length $\PL{u}(n)$.

\begin{prop}\label{prop:upper_est_on_pall}
Let $\psi:A^*\to A^*$ be a morphism in class $\Pc$ such that for each $a\in A$ it holds that
$\psi(a)=pq_a$, where $p\in\{\varepsilon\}\cup A$ and $q_a$ is a palindrome.
Let us denote
\[
C:=\max\{\pall{x}:\exists\,a\in A, \text{$x$ is a proper prefix of $q_a$}\}.
\]
Then for a fixed point $\bu$ of $\psi$ we have
\[
\limsup_{n\to\infty}\frac{\PL{u}(n)}{\ln n} \leq \frac{2C+\frac{3}{2}|p|}{\ln\Lambda},
\]
where $\Lambda$ is the dominant eigenvalue of the incidence matrix of $\psi$.
\end{prop}

\begin{proof}
  First realize that under our assumptions, the set $\{\pall{x}:\exists\,a\in A, \text{$x$ is a proper prefix of $q_a$}\}$
  is non-empty. Otherwise, $\psi(a)=p$ for every letter $a$, and as $|p|\leq 1$, the morphism $\psi$ is not primitive (and
  thus not in class $\Pc$). Therefore the constant $C$ is well defined. Moreover, note that it is enough to consider only
  prefixes (in the definition of $C$) since $q_a$ are palindromes and thus if $x$ is a suffix of $q_a$ then
  $\overleftarrow{x}$ is its prefix and $\pall{x}=\pall{\overleftarrow{x}}$.

  Consider a fixed point $\bu$ of $\psi$. If it is eventually periodic, then by~\cite{frid-puzynina-zamboni-aam-50}, the
  palindromic length of its factors is bounded and the statement of the proposition is trivially valid. Assume that $\psi$
  has an aperiodic fixed point. Since $\psi$ is a primitive morphism, every sufficiently long factor of $\bu$ has a uniquely
  determined preimage~\cite{mosse-bsmf-124}. More precisely, there exists $n_0\in\N$ such that for each
  $v\in\Lc(\bu)$, $|v|>n_0$, there are factors $x,v',y$ of $\bu$ such that $v=x\psi(v')y$, where
  $v'\neq\varepsilon$, $x$ is a proper suffix of $\psi(a)$ and $y$ is
  a proper prefix of $\psi(b)$ for some letters $a,b\in A$, cf. Figure~\ref{fig:phiimage}.

  \begin{figure}[!htp]
    \bigskip
    \centering
    \hspace{-2em}\includegraphics[scale=0.8]{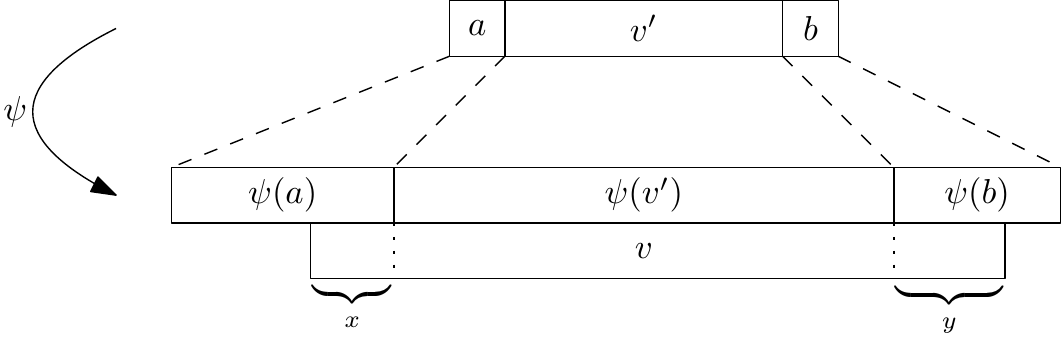}
    \caption{Illustration of the construction of preimage of $v$ in the proof of Proposition~\ref{prop:upper_est_on_pall}.}\label{fig:phiimage}
  \end{figure}

  Obviously, $\pall{v}\leq\pall{x}+\pall{\psi(v')}+\pall{y}$.
  By definition of $C$, we have $\pall{x}\leq C$, since $\psi(a)=pq_a$ we have
  $\pall{y}\leq C+|p|$. Using Lemma~\ref{lem:pall_of_image_for_p_shorter_than_2}, $\pall{\psi(v')}\leq\pall{v'}+\delta$,
  where $\delta=1$ if $\pall{v'}$ is odd and $|p|=1$, and $\delta=0$ otherwise. Together, we obtain
  \[
  \pall{v}\leq 2C+|p|+\pall{v'}+\delta.
  \]
  If $|v'|>n_0$ we apply the same procedure to $v'$. In this way, for a given  $v\in\Lc(\psi)$  we create a sequence
  $v=v^{(0)},v^{(1)},\ldots,v^{(k)}$ such that for each $i=1,2,\ldots,k$ we have
  \begin{gather}
    \pall{v^{(i-1)}} \leq 2C+|p|+\pall{v^{(i)}} + \delta_i, \label{eq:pall_in_seq_v} \\
    |v^{(i-1)}|\geq |\psi(v^{(i)})|, \label{eq:len_in_seq_v}
  \end{gather}
  and $|v^{(k)}|\leq n_0$.
  From~(\ref{eq:pall_in_seq_v}) we get
  \begin{equation}\label{eq:horni}
    \pall{v} \leq 2kC + n_0 + |p|\Big(k+\Big\lceil\frac{k}{2}\Big\rceil\Big),
  \end{equation}
  where we used that $|v^{(k)}|\leq n_0$, and iii) of Lemma~\ref{lem:pall_of_image_for_p_shorter_than_2}.

  On the other hand,~\eqref{eq:len_in_seq_v} implies that
  \begin{equation}\label{eq:dolni}
    |v| \geq |\psi^k(v^{(k)})| = \boldsymbol{1}\cdot \boldsymbol{M}_{\!\psi}^{k}\cdot\vec{V}(v^{(k)}).
  \end{equation}
  Since $\psi$ is a primitive morphism, by the Perron-Frobenius theorem, the dominant eigenvalue $\Lambda$ of the
  matrix $M_\psi$ is positive and strictly greater that absolute values of all other eigenvalues of $M_\psi$. Denote by
  $R$ the non-singular matrix such that $RM_\psi R^{-1}$ is in the Jordan canonical form, i.e., it is block diagonal.
  Notably, the block corresponding to the eigenvalue $\Lambda$ is of dimension $1\times 1$. Without loss of generality,
  let it be the first block on the diagonal of $RM_\psi R^{-1}$. Then necessarily,
  \[
  \lim_{k\to\infty} \frac1{\Lambda^k} M_\psi^k = R\begin{pmatrix}
                                             1 & 0 & \cdots & 0 \\
                                             0 & 0 & \cdots & 0 \\
                                             \vdots & \vdots & \ddots & \vdots \\
                                             0 & 0 & \cdots & 0
                                           \end{pmatrix} R^{-1}.
  \]
  Combining with~\eqref{eq:dolni}, we have that
  \[
  |v| \geq |\psi^k(v^{(k)})| = c_1\Lambda^k(1+o(1)).
  \]
  Putting the latter estimate together with~\eqref{eq:horni},
  \[
  \frac{\pall{v}}{\ln|v|} \leq
  \frac{2kC + n_0 + |p|\Big(k+\Big\lceil\frac{k}{2}\Big\rceil\Big)}
       {\ln c_1+k\ln\Lambda+\ln(1+o(1))}.
  \]
  If $|v|=n$ tends to infinity then $k$ tends to infinity as well. The validity of the proposition follows.
\end{proof}

Proposition~\ref{prop:upper_est_on_pall} provides an upper estimate on the palindromic length $\PL{u}(n)$ for any fixed
point $\bu$ of any morphism of class $\Pc$. The result is valid independently of the size of the alphabet. Reducing our
consideration to binary infinite words, we recall the result of Bo Tan~\cite{tan-tcs-389}. He shows that any binary
morphism $\varphi$ producing a fixed point with infinitely many palindromes is either itself conjugated to a morphism in
class $\Pc$, or this can be said about its second iterate $\varphi^2$. This allows us to formulate a summarizing
corollary to our Proposition~\ref{prop:upper_est_on_pall}.

\begin{thm}\label{thm:hlavni}
  Let $\bu$ be a fixed point of a primitive morphism over a binary alphabet. Then there is a constant
  $K>0$ such that either
  \[
  \PL{u}(n)\leq K\ln n\qquad \text{for all $n\in\N$},
  \]
  or
  \[
  \PL{u}(n)\geq Kn\qquad \text{for all $n\in\N$}.
  \]
\end{thm}
\begin{proof}
  Let $\Lc(\bu)$ contain only finitely many palindromes. Then obviously for every factor $w$ of length $n$, we have
  $\pall{w}\geq\frac{1}{c}n$, where $c$ is the length of the longest palindrome in $\Lc(\bu)$. Thus $\PL{u}(n)$ grows
  at least linearly.

  On the other hand, if $\Lc(\bu)$ contains infinitely many palindromes then it coincides with
  the language $\Lc(\psi)$ for some morphism $\psi$ in class $\Pc$ (cf.~\cite{tan-tcs-389}) and
  thus we can employ Proposition~\ref{prop:upper_est_on_pall}.
\end{proof}

Note that on morphisms over an alphabet with more than two letters, one cannot prove a result as strong as that of
Bo Tan~\cite{tan-tcs-389}. A counterexample was given on a ternary morphism in~\cite{labbe-ejc-21}.

%%%%%%%%%%%%%%%%%%%%%%%%%%%%%%%%%%%%%%%%%%%%%%%%%%%%%%%%%%%%%%%%%%%%%%%%%%%%%%%%%%%%%%%%%%%%%%
\section{Fibonacci and Thue-Morse words}\label{sec:FibThu}

Let us provide an upper bound on the constant $K$ of Theorem~\ref{thm:hlavni} for the Fibonacci word $\boldsymbol{f}$
and for the Thue-Morse word $\boldsymbol{t}$.

\paragraph{The Fibonacci word}
The Fibonacci morphism
$\varphi_F:a\mapsto ab,b\mapsto a$ belongs to class $\Pc$ (cf.\ Example~\ref{ex:fibmorphism}). Since its fixed point,
the Fibonacci word $\boldsymbol{f}$, is a Sturmian word, it follows from the result by Frid~\cite{frid-ejc-71} that
$\limsup_{n\to\infty}\PL{f}(n)=+\infty$.

Let us apply Proposition~\ref{prop:upper_est_on_pall} in this case. Obviously $p=a$ and $C=0$
and the dominant eigenvalue of the incidence matrix $\left(\begin{smallmatrix}1 & 1 \\ 1 & 0\end{smallmatrix}\right)$
of $\varphi_F$ is the golden mean $\tau=\frac{1+\sqrt{5}}{2}$. Therefore
\begin{equation}\label{eq:fibonacci_upper_est_1}
  \limsup_{n\to\infty}\frac{\PL{f}(n)}{\ln n} \leq \frac{3}{2\ln\tau}.
\end{equation}
The Fibonacci word $\boldsymbol{f}$ is also the fixed point of $\varphi_F^3:a\mapsto abaab,b\mapsto aba$.
Consider morphism $\psi: a\mapsto ababa,b\mapsto aba$. Taking $w=aba$, we see that
\[
\varphi_F^3(a)w=abaababa=w\psi(a),\qquad \varphi_F^3(b)w=abaaba=w\psi(b),
\]
which by Definition~\ref{d:conjugace} means that $\varphi_F^3\sim\psi$. According to Proposition~\ref{p:conjuglang},
we have $\Lc(\varphi_F^3)=\Lc(\psi)$. Let us apply Proposition~\ref{prop:upper_est_on_pall} to $\psi$.
Obviously $p=\varepsilon$ and $C=2$. The dominant eigenvalue of $\boldsymbol{M}_{\!\psi}$ is $\tau^3$.
Therefore
\begin{equation}\label{eq:fibonacci_upper_est_2}
  \limsup_{n\to\infty}\frac{\PL{f}(n)}{\ln n} \leq \frac{4}{3\ln\tau},
\end{equation}
which gives a better estimate than~(\ref{eq:fibonacci_upper_est_1}).

If we use $\psi^2: a\mapsto ababaabaababaabaababa,b\mapsto ababaabaababa$, which also fixes the
Fibonacci word, we have $p=\varepsilon$, $C=3$, and $\Lambda=\tau^6$. This improves the constant
in estimate~(\ref{eq:fibonacci_upper_est_2}) to $\frac{1}{\ln\tau}$. Making similar considerations
for $\psi^k$, $k\leq 13$, we obtain that $K\leq \frac{2(k+1)}{3 k \ln\tau}$. This makes us conjecture
that
\[
\limsup_{n\to\infty}\frac{\PL{f}(n)}{\ln n}=\frac{2}{3\ln\tau}.
\]

Let us remark that Frid~\cite{frid-numeration-2018} investigated the palindromic length only of
prefixes of the Fibonacci word. She conjectures that the prefix $w^{(k)}$ (of the Fibonacci word)
whose length written in the Zeckendorf numeration system %~\cite{zeckendorf-bsrsl-41}
is $(|w^{(k)}|)_F=(100)^{2k-1}101$ has $\pall{w^{(k)}}=2k+1$. Should this conjecture be valid,
it would imply that
\[
\limsup_{n\to\infty}\frac{\PL{f}(n)}{\ln n} \geq \frac{1}{3\ln\tau}.
\]

\paragraph{The Thue-Morse word}
Let us consider the Thue-Morse word $\boldsymbol{t}$, i.e., the fixed point of the morphism
$\varphi_{TM}^2: a\mapsto abba, b\mapsto baab$ (cf. Example~\ref{ex:tmmorphisms}). Similarly to the case of
the Fibonacci word we are interested in the constant $K$ where $\limsup_{n\to\infty}\frac{\PL{t}(n)}{\ln n}\leq K$.

Applying Proposition~\ref{prop:upper_est_on_pall} to $\varphi_{TM}^2$, where $\Lambda=4$, $C=2$, $p=\varepsilon$,
we get $K \leq \frac{4}{\ln 4}$.
Further iteration of the procedure in Proposition~\ref{prop:upper_est_on_pall} with $\varphi_{TM}^{2k}$ for $k\leq 13$
show that $K\leq (3+\frac{1}{k})\frac1{\ln 4}$. This leads us to conjecture that
\[
\limsup_{n\to\infty}\frac{\PL{t}(n)}{\ln n} = \frac{3}{\ln 4}.
\]

%%%%%%%%%%%%%%%%%%%%%%%%%%%%%%%%%%%%%%%%%%%%%%%%%%%%%%%%%%%%%%%%%%%%%%%%%%%
\section{Words rich in palindromes}\label{sec:rich}

Intuitively, a word containing many palindromic factors should have small palindromic length. Recall
that Droubay et al.~\cite{droubay-justin-pirillo-tcs-255} found out that a finite word $w$ contains at
most $|w|+1$ different palindromic factors. If this bound is attained, the word $w$ is called rich (in palindromes).
An infinite word $\bu$ is called rich if each of its factors is rich. The difference between the upper bound $|w|+1$
and actual number of palindromes in $w$ is called the (palindromic) defect of $w$ and denoted by $D(w)$,
see~\cite{brlek-hamel-nivat-reutenauer-ijfcs-15}.

\begin{ex}
  Word $w_1=abbabaab$ contains 9 palindromic factors: $\varepsilon$, $a$, $b$, $aa$, $bb$, $aba$, $bab$, $abba$, $baab$,
  and thus $w_1$ is rich. On the other hand, the word $w_2=bbabaabb$ contains only 8 palindromic factors:
  $\varepsilon$, $a$, $b$, $aa$, $bb$, $aba$, $bab$, $baab$, and thus $w_2$ is not rich. And so the Thue-Morse word $\boldsymbol{t}$
  is not rich, since $w_2$ is one of its factors. The defect of $w_2$ is $D(w_2)=1$. It follows from the results
  of~\cite{blondin-masse-brlek-garon-labbe-puma-19} that $\sup_{w\in\Lc(\boldsymbol{t})}D(w)=+\infty$.
\end{ex}

Some of the rich words have small palindromic length. For example, the palindromic length of morphic Sturmian words such
as the Fibonacci word $\boldsymbol{f}$ grows at most logarithmically (cf. Theorem~\ref{thm:hlavni}). Nevertheless, quite
surprisingly, richness of a word does not imply that its palindromic length be small. We will demonstrate this fact on finite
words
\[
w^{(k)}=abaabb\cdots a^kb^k,\qquad v^{(k)}=abbaaabbbb\cdots a^{2k-1}b^{2k}.
\]
In all these words
the sequences of runs of $a$ and of runs of $b$ are monotone, and therefore $w^{(k)}$ and $v^{(k)}$ are rich
for all $k\in\N$ by a result of Guo et al.~\cite{guo-shallit-shur-ipl-116}.

\begin{prop}\label{prop:pal_len_of_rich_words_wk_and_uk}
  For every $k\in\N$ we have $\pall{w^{(k)}}=k+1$ and $\pall{v^{(k)}}=k+1$.
\end{prop}

\begin{proof}
  A palindromic factor in $w^{(k)}$ has one of the following forms
  $a^i$, $b^i$, $a^ib^ja^i$, $b^ia^jb^i$ for some $i,j\in\N$. Obviously, a minimal
  factorization will contain as many as possible blocks of the 3rd and 4th type.
  One can use at most  $k-1$ such palindromes (since there are $2k-1$ runs of
  zeroes and ones in $w^{(k)}$, excluding the outer ones) plus at least two palindromes
  of the type $a^i$ or $b^i$ to cover whole $w^{(k)}$. Examples of such minimal
  palindromic factorizations follows.
  \begin{align*}
    w^{(6)} & = aba\cdot abba\cdot aabbbaa\cdot aabbbbaa\cdot aaabbbbbaaa\cdot aaa\cdot bbbbbb = \\
           & = aba\cdot abba\cdot aa \cdot bbbaaaabbb\cdot baaaaab\cdot bbbbaaaaaabbbb\cdot bb = \\
           & = a\cdot baab \cdot baaab \cdot bbaaaabb\cdot bbaaaaabb\cdot bbbaaaaaabbb\cdot bbb.
  \end{align*}
  The same consideration holds for $v^{(k)}$.
\end{proof}

With the use of the above proposition and a result on rich words derived in~\cite{rukavicka-dlt-2017}, we can summarize
the following information about palindromic length of finite rich words.

\begin{coro}\label{c:rich}
  There are constants $c>0$, $d>1$ such that
  \begin{enumerate}
  \item
    $\pall{w}\geq c\sqrt{|w|}$ for infinitely many finite rich words $w$,
  \item
    $\pall{w}\leq d\frac{|w|}{\ln|w|}$ for every rich word $w$.
  \end{enumerate}
\end{coro}

\begin{proof}
  i) By definition, the length of the word $w^{(k)}$ is $|w^{(k)}|=k(k+1)$. Using
  Proposition~\ref{prop:pal_len_of_rich_words_wk_and_uk}, we have $\pall{w^{(k)}}=k+1 > \sqrt{k(k+1)}$.
  Similarly, we can use the words $v^{(k)}$.\\
  ii) This is the statement of Theorem 3 in~\cite{rukavicka-dlt-2017}.
\end{proof}

The finite words $w^{(k)}$ ($v^{(k)}$ resp.), $k\in\N$, allow one to define infinite rich words.

\begin{prop}\label{p:rich}
  There exists an infinite word $\bu$ rich in palindromes for which
  \[
  \sqrt{n}-3 \leq \PL{u}(n) \leq \sqrt{n}+4 \qquad \text{for every $n\in\N$}.
  \]
  %% \[
  %% c\sqrt{n}\leq \PL{u}(n)\leq d\frac{n}{\ln n},
  %% \]
  %% for some constants $c,d>0$.
\end{prop}

\begin{proof}
  It suffices to define the infinite word $\bu$ as the word having the prefix $w^{(k)}$ for every $k\in\N$.
  %% By~\cite{guo-shallit-shur-ipl-116}, the word $\bu$ is rich. Using Corollary~\ref{c:rich}, we obtain the result.
  
  Let $u$ be a prefix of $\bu$ of length $n$, then we show that
  \begin{equation}\label{eq:pref_rich_u}
    \sqrt{n}-3\leq\pall{u}\leq\sqrt{n}+4.
  \end{equation}
  Let $k$ be such that
  \begin{equation}\label{eq:pref_rich_uu}
    k(k+1)<n\leq(k+1)(k+2).
  \end{equation}
  Then
  \begin{itemize}
  \item
    since $u$ has prefix $w^{(k)}$, we have $u=w^{(k)}z$, where $z=a^l$ or $z=a^{k+1}b^l$ for some $l\leq k+1$.
    Thus $\pall{z}\leq 2$,
  \item
    $w^{(k+1)}$ has prefix $u$ and thus $w^{(k+1)}=uv$, where (analogously to the previous case) $\pall{v}\leq 2$.
  \end{itemize}
  Lemma 6 from~\cite{saarela-lncs-10432} states that $\pall{x}\leq\pall{y}+\pall{xy}$. Using this lemma for
  $x=w^{(k)}$ and $y=z$ and then for $x=u$ and $y=v$ we get
  \[
  k-1 = \pall{w^{(k)}}-2 \leq \pall{u} \leq \pall{w^{(k+1)}}+2 = k+4. 
  \]
  Since~(\ref{eq:pref_rich_uu}) implies $k<\sqrt{n}<k+2$ the estimate~(\ref{eq:pref_rich_u}) follows.

  Analogously to the proof of Proposition~\ref{prop:pal_len_of_rich_words_wk_and_uk}, the maximum of the set
  $\{\pall{u}:|u|=n,u\in\mathcal{L}(\bu)\}$ is reached on the prefixes of $\bu$. Therefore
  \[
  \sqrt{n}-3 \leq \PL{u}(n) \leq \sqrt{n}+4. \qedhere
  \]
\end{proof}

%%%%%%%%%%%%%%%%%%%%%%%%%%%%%%%%%%%%%%%%%%%%%%%%%%%%%%%%%%%%%%%%%%%%%%%%%%%
\section{Combined pal-antipal length}\label{sec:antipal}

Let us reconsider the Thue-Morse word. Besides infinitude of palindromes, it contains also infinitely many the so-called 
antipalindromes~\cite{blondin-masse-brlek-garon-labbe-puma-19}. These words have been considered in a wider context under 
the name $f$-pseudo-palindromes (or $f$-palindromes) already 
in~\cite{anne-zamboni-zorca-words-2005,deluca-deluca-tcs-362,halava-harju-karki-zamboni-tucs-839}.
We use the name antipalindrome in accordance with~\cite{guo-shallit-shur-unpbl-2015}.

A finite word $w=w_1\cdots w_n$ over a binary alphabet
$\{a,b\}$ is an antipalindrome, if $w=E(w_n)E(w_{n-1})\cdots E(w_1)$, where $E$ exchanges the letters, $E(a)=b$, $E(b)=a$.
Obviously, an antipalindrome is always of even length. A finite word thus need not to be factorizable into only
antipalindromes.

An extension of the question on palindromic length could be on the factorization of a given finite word into the smallest
possible number of factors which are either palindromes of antipalindromes. For that purpose, we have adapted the simple
quadratic algorithm for minimal palindromic factorization given in~\cite{fici-gagie-karkkainen-kempa-jda-28}. We have
computed the palindromic and combined pal-antipal length for the prefixes of the Thue-Morse word $\boldsymbol{t}$
of length up to $10^6$, see graph of $\PL{t}(n)/\ln n$ in Figure~\ref{f}.

\begin{figure}[!ht]
  \centering
  \includegraphics[width=0.47\textwidth]{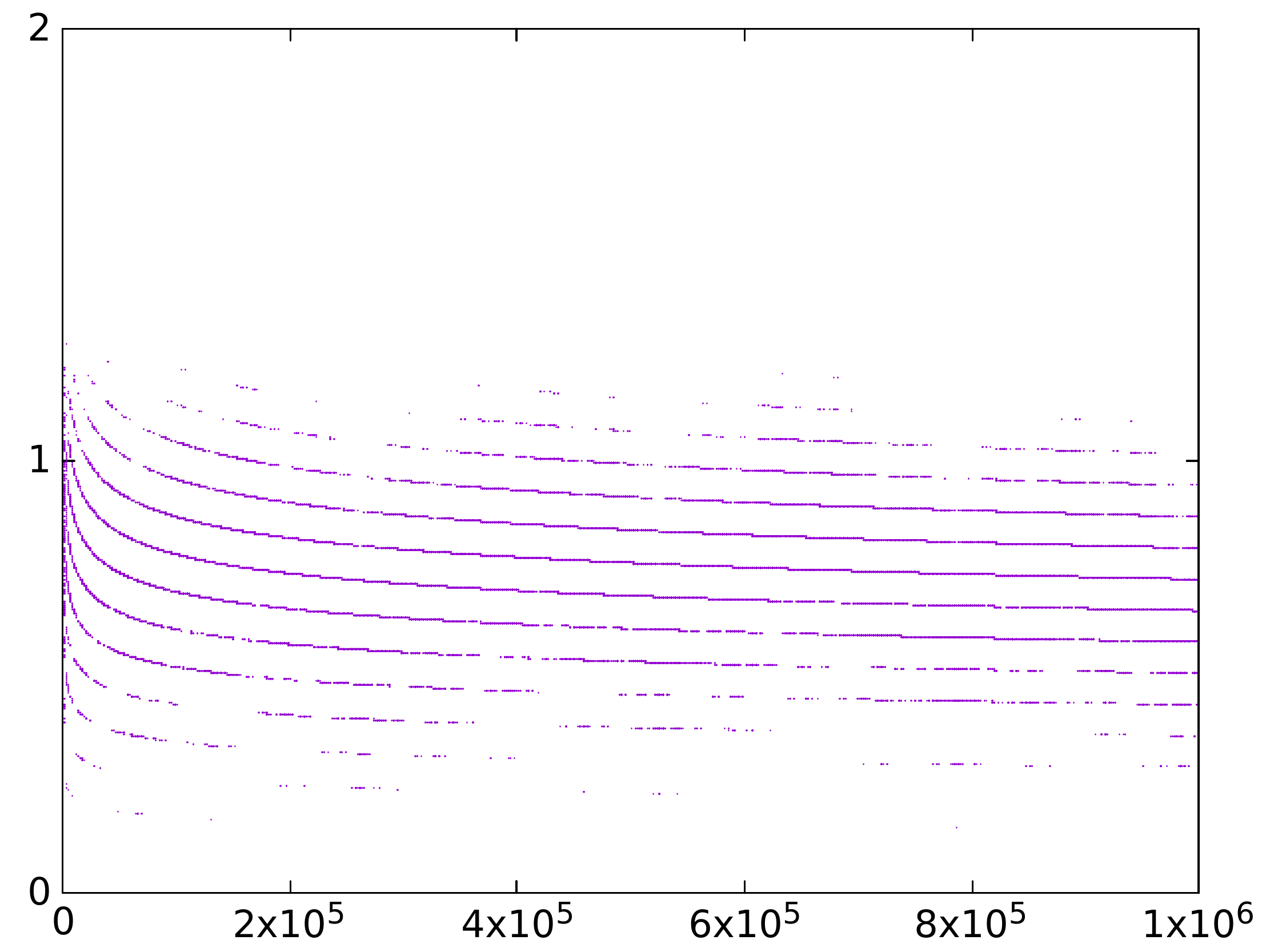}
  \quad
  \includegraphics[width=0.47\textwidth]{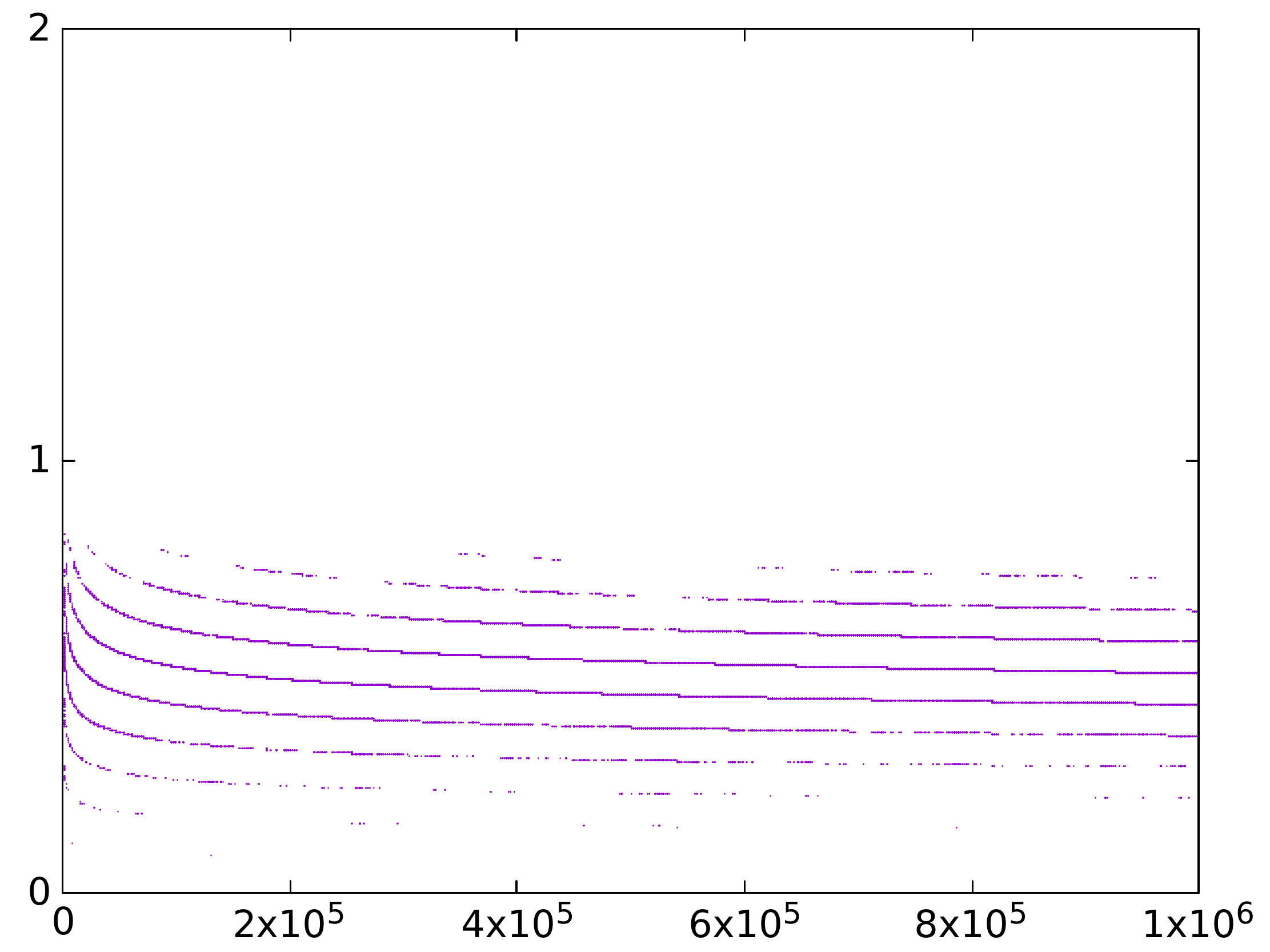}
  \caption{On the left, $\PL{t}(n)/\ln n$ for $1\leq n\leq 10^6$ is displayed.
    The same for the combined pal-antipal length on the right.}\label{f}
\end{figure}

In Tables~\ref{t:double} and~\ref{t:same} we include some of the results of our computations comparing the palindromic length $\pall{w}$ and 
the combined pal-antipal length $\pallc{w}$ of the prefixes $|w|$ of the Thue-Morse word $t$.

\begin{table}[!ht]
\centering
\begin{tabular}[t]{rcc}
    \toprule
    $|w|$ & $\pall{w}$ & $\pallc{w}$ \\ \midrule[\heavyrulewidth]
    2 & 2 & 1  \\ \midrule
    10 & 4 & 2  \\ \midrule
    118 & 6 & 3  \\ \midrule
    630 & 8 & 4  \\ \midrule
    7542 & 10 & 5  \\ 
    \bottomrule
\end{tabular}
\hspace{2em}
\begin{tabular}[t]{rcc}
    \toprule
    $|w|$ & $\pall{w}$ & $\pallc{w}$ \\ \midrule[\heavyrulewidth]
    40310 & 12 & 6  \\ \midrule
    482678 & 14 & 7  \\ \midrule
    2579830 & 16 & 8 \\ \midrule
    30891382 & 18 & 9 \\ 
    \bottomrule
\end{tabular}
\caption{Combined pal-antipal length $\pallc{w}$ of prefixes $w$ of the Thue-Morse word $\boldsymbol{t}$ can be up to twice smaller than their 
palindromic length $\pall{w}$. For $1\leq n\leq 9$, we display the first occurrences of the case when $\pall{w}=2\pallc{w}=2n$.}\label{t:double} 
\end{table}

Our computations suggest that including both palindromes and antipalindromes into the factorization of Thue-Morse word,
one can reduce the number of factors by half (and not more), see Table~\ref{t:double}.
On the other hand, we suppose that there exist prefixes of arbitrary palindromic length whose combined pal-antipal length
as big as the palindromic one, see~Table~\ref{t:same}. Other research could be done in this direction.

\begin{table}[!ht]
\centering
  \begin{tabular}{rcc}
    \toprule
    $|w|$ & $\pall{w}$ & $\pallc{w}$   \\ \midrule[\heavyrulewidth]
    1 & 1 & 1  \\ \midrule
    3 & 2 & 2  \\ \midrule
    11 & 3 & 3 \\ \midrule
    43 & 4 & 4 \\ \midrule
    171 & 5 & 5 \\ \midrule
    683 & 6 & 6 \\ \midrule
    2731 & 7 & 7 \\ \bottomrule
  \end{tabular}
  \hspace{2em}
  \begin{tabular}{rcc}
    \toprule
    $|w|$ & $\pall{w}$ & $\pallc{w}$   \\ \midrule[\heavyrulewidth]
    10923 & 8 & 8  \\ \midrule
    43691 & 9 & 9  \\ \midrule
    174763 & 10 & 10 \\ \midrule
    699051 & 11 & 11  \\ \midrule
    2796203 & 12 & 12  \\ \midrule
    11184811 & 13 & 13  \\ \midrule
    44739243 & 14 & 14 \\ \bottomrule
   \end{tabular}
  \caption{The prefixes of the infinite word $\boldsymbol{t}$ on which considering antipalindromes does not improve the palindromic length.
  For each $1\leq n\leq 14$, we display the first occurrences of the case when $\pall{w}=\pallc{w}=n$ happens.}\label{t:same} 
\end{table}

%%%%%%%%%%%%%%%%%%%%%%%%%%%%%%%%%%%%%%%%%%%%%%%%%%%%%%%%%%%%%%%%%%%%%%%%%%%
\section{Open problems}

The palindromic length of finite and infinite words has been introduced in 2013~\cite{frid-puzynina-zamboni-aam-50}. Since then,
several groups of authors focused on the design of fast algorithms for computing the minimal palindromic factorization, see
e.g.~\cite{fici-gagie-karkkainen-kempa-jda-28,rubinchik-shur-lncs-9538,borozdin-kosolobov-rubinchik-shur-lipi-78}. On the other hand,
an analytic study of the palindromic length is still in its beginnings. Let us formulate several open questions which we consider of interest.

\begin{enumerate}
\item
  When studying palindromic length of the Fibonacci and Thue-Morse word, we have conveniently considered
  a power of the morphism that could be conjugated to a morphism in which the image of every letter is a palindrome, i.e.,
  $a\mapsto q_a$ for every $a\in A$. According to our knowledge, question on determining for which morphisms
  such a power exists, has not been considered yet.
\item
  In the study of the growth of $\PL{u}(n)$ we provide a method of finding an upper bound on the constant $K$, in the estimate
  $\PL{u}(n)\leq K\ln n$, which is valid for any fixed point $\bu$ of any morphism in class $\Pc$. According to our knowledge, no methods
  for giving a lower bound on $\PL{u}(n)$ have been mentioned in the literature. So far, only Frid~\cite{frid-numeration-2018} has
  focused on finding a lower bound on the palindromic length. Her study is specific for the Fibonacci word $\boldsymbol{f}$. She
  states a conjecture describing the prefixes $w$ of $\boldsymbol{f}$ having palindromic length $\pall{w}$ strictly bigger than all
  the shorter prefixes of $\boldsymbol{f}$.
\item
  The validity of the conjectured lower bound of Frid~\cite{frid-numeration-2018} would imply that
  $\limsup_{n\to\infty}\PL{f}(n)/\ln n\geq 1/(3\ln\tau)$. Our computations (cf.\ Section~\ref{sec:FibThu}) suggest that $\limsup$ should have
  a bigger value. This is probably caused by the fact that Frid only considers the palindromic length of prefixes of the Fibonacci word.
  It may be the case that bigger palindromic length is achieved on factors that are not prefixes of $\boldsymbol{f}$. We do not have
  candidates for such factors. It should be mentioned that Saarela~\cite{saarela-lncs-10432} shows equivalence between the unboundedness
  of the palindromic length when taken over the factors and considering only the prefixes. This, however, does not mean that the growth
  of the function dependingly on $n$ should be equal.
\item
  In Proposition~\ref{p:rich}, we give an infinite rich word whose palindromic length grows with $n$ at least as $\sqrt{n}$. The infinite
  word is however not uniformly recurrent. All the other considered classes of palindromic uniformly recurrent words have palindromic
  length bounded by $K\ln n$. Does there exist a uniformly recurrent infinite word $\bu$ such that $\PL{u}(n)\geq c\sqrt{n}$?
\item
  HKS conjecture was formulated in view of characterization of morphisms providing fixed points with infinitely many palindromes.
  It is not obvious which morphisms generate fixed points that besides infinitely many palindromes contain also arbitrarily long
  antipalindromes, as it is the case of the Thue-Morse morphism.
\end{enumerate}

%%%%%%%%%%%%%%%%%%%%%%%%%%%%%%%%%%%%%%%%%%%%%%%%%%%%%%%%%%%%%%%%%%%%%%%%%%%%%%%%%%%%%%%%%%%%%%%%%%%%%%

\section*{Acknowledgments}

\noindent
This work was supported by the project CZ.02.1.01/0.0/0.0/16\_019/0000778
from European Regional Development Fund. We also acknowledge financial support of the Grant Agency of the
Czech Technical University in Prague, grant No.\ SGS14/205/OHK4/3T/14.

%%%%%%%%%%%%%%%%%%%%%%%%%%%%%%%%%%%%%%%%%%%%%%%%%%%%%%%%%%%%%%%%%%%%%%%%%%%%%%%%%%%%%%%%%%%%%%%%%%%%%%

%\bibliography{./reference}
%\bibliographystyle{elsarticle-num-alphsort}

\end{document}